\documentclass[12pt]{amsart} 
\usepackage{latexsym,mathtools}
\usepackage{epsfig,setspace}
\usepackage[utf8]{inputenc}
\usepackage{tipa}
\usepackage{a4}
\usepackage{amssymb}
\usepackage{amsthm}
\usepackage{amsbsy}
\usepackage{upgreek}
\usepackage{amsfonts,amssymb,latexsym,amscd,amsmath,euscript,enumerate,verbatim,calc} 
\allowdisplaybreaks[1]
\usepackage{url}
\usepackage{hhline}
\usepackage{pifont}
\usepackage[colorlinks=true,linkcolor=blue,citecolor=magenta]{hyperref}
\theoremstyle{plain}

\newtheorem{theorem}{Theorem}[section]

\newtheorem{proposition}[theorem]{Proposition}

\newtheorem{problem}[theorem]{Problem}
\newtheorem{conjecture}[theorem]{Conjecture}
\newtheorem{corollary}[theorem]{Corollary}
\theoremstyle{definition}
\newtheorem{definition}[theorem]{Definition}

\newtheorem{remark}[theorem]{Remark}

\def\Fq{{\mathbb F}_q}

\def\Bk{\mathcal{B}_k}
\def\Bn{\mathcal{B}_n}

\newcommand{\rank}{\operatorname{rank}}
\newcommand{\GL}{\operatorname{GL}}

\makeatletter
\def\imod#1{\allowbreak\mkern10mu({\operator@font mod}\,\,#1)}
\makeatother

\title{Counting zero kernel pairs over a finite field}
\author{Samrith Ram} 
\address{Department of Mathematics, Indian Institute of Science, Bangalore  560012, India.}
\email{samrith@gmail.com}
\date{\today}



\begin{document} 
\begin{abstract}
Helmke et al. have recently given a formula for the number of reachable pairs of matrices over a finite field. We give a new and elementary proof of the same formula by solving the equivalent problem of determining the number of so called zero kernel pairs over a finite field. We show that the problem is equivalent to certain other enumeration problems and outline a connection with some recent results of Guo and Yang on the natural density of rectangular unimodular matrices over $\Fq[x]$. We also propose a new conjecture on the density of unimodular matrix polynomials.
\end{abstract}
\maketitle
\section{Introduction}
Let $\Fq$ denote the finite field with $q$ elements. %and let $\Fq[x]$ denote the ring of polynomials in one variable $x$ with coefficients in $\Fq$. 
For positive integers $n,k$, we denote by $M_{n,k}(\Fq)$ the set of all $n \times k$ matrices with entries in $\Fq$ and by $M_n(\Fq)$ the set of all square $n\times n$ matrices with entries in $\Fq$. Throughout this paper we assume $k<n$ unless otherwise stated. Consider the following problems:
\begin{problem}
\label{howmanycompletions}
How many matrices in $M_{n,k}(\Fq)$ occur as the submatrix formed by the first $k$ columns of some matrix in $M_n(\Fq)$ with irreducible characteristic polynomial?  
\end{problem}
\begin{problem}
\label{howmanyunimodular}
 How many matrices $Y\in M_{n,k}(\Fq)$ have the property that the linear matrix polynomial
  \begin{equation}
    \label{eq:AC}
    x{I_k \brack \bf{0}}-Y
  \end{equation}
is unimodular (i.e. has all invariant factors equal to 1)? ($I_k$ denotes the $k\times k$ identity matrix, $\bf{0}$ denotes the $n-k \times k$ zero matrix).
\end{problem}
\begin{problem}
\label{numberofcpairs}
  How many pairs of matrices $(A,B)\in M_k(\Fq)\times M_{k,n-k}(\Fq)$ have the property that 
  \begin{equation}
\label{controllabilitymatrix}
\rank  \begin{bmatrix}
B &AB &\cdots & A^{k-1}B
  \end{bmatrix}=k?
  \end{equation} 
\end{problem}
\begin{problem}
\label{howmanyT}
 If $V$ is an $n$-dimensional vector space over $\Fq$ and $W$ is a fixed $k$-dimensional subspace of $V$, how many linear transformations $T:W\to V$ have the property that the only $T$-invariant subspace (contained in $W$) is the zero subspace?  
\end{problem}
% \item How many pairs of matrices $(A,C)\in M_k(\Fq)\times M_{n-k,k}(\Fq)$ have the property that the polynomial matrix
%   \begin{equation}
%     \label{eq:AC}
%     {xI_k-A \brack -C} 
%   \end{equation}
% is unimodular? 
Interestingly, all the above problems are equivalent and have the same answer given by $\prod_{i=1}^{k}(q^n-q^i)$. Problem \ref{numberofcpairs} was considered by Koci\textpolhook{e}cki and Przyłuski \cite{KocPrz1989} in the context of estimating the proportion of reachable linear systems over a finite field. They gave an explicit answer to Problem \ref{numberofcpairs} in the cases $n-k=1,2$. In the same paper, they stated that the general problem ``seems to be rather difficult''. In fact, the problem of finding an explicit formula for the number of `reachable pairs' $(A,B)$ (i.e. pairs of matrices satisfying \eqref{controllabilitymatrix}) has been settled only very recently by Helmke et al. \cite[Thm. 1]{Helmkeetal2015}. The proof relies on some earlier results by Helmke \cite{Helmke1986,Helmke1993} and uses some advanced geometric techniques.% involves counting $\Fq$-rational points on a suitably defined quasi-affine algebraic variety using a cell decomposition obtained by fixing the Hermite indices of controllable pairs (also referred to as reachable pairs). 

 In this paper, we give a new proof of the same formula for the number of reachable pairs by first showing in Section \ref{background} that Problems \ref{howmanycompletions}, \ref{howmanyunimodular}, \ref{numberofcpairs} above are indeed equivalent. In Section \ref{enumeration} we explain the connection with Problem \ref{howmanyT} which we subsequently solve. Our proof is self-contained and uses only elementary methods in linear algebra and the $q$-Vandermonde identity for Gaussian binomial coefficients. We also highlight a connection between Problem  \ref{howmanyunimodular} and some recent results by Guo and Yang \cite{GuoYang2013} on the natural density of rectangular unimodular matrices over $\Fq[x]$. We then propose a new conjecture on the density of unimodular matrices that generalizes the problems stated earlier in the introduction.
  
\section{Background}
\label{background}
 Problem~\ref{howmanycompletions} is in fact a combinatorial matrix completion problem -- we would like to count the number of matrices in $M_{n,k}(\Fq)$ for which there exists a completion (by padding $n-k$ columns on the right) to a square matrix in $M_n(\Fq)$ with irreducible characteristic polynomial. Our starting point is the following result of Wimmer \cite{Wimmer1974} (also see Cravo \cite[Thm. 15]{Cravo2009}).
\begin{theorem}[Wimmer]
\label{wimmer}
Let $F$ be an arbitrary field and let $A\in M_k(F), C\in M_{n-k , k}(F)$. Suppose $f(x)\in F[x]$ is a monic polynomial of degree $n$ and let $f_1(x)\mid \cdots \mid f_k(x)$ be the invariant factors of the polynomial matrix
\begin{equation}
  \label{xI-kcols} 
\begin{bmatrix}
xI_k-A \\
-C
\end{bmatrix}.
\end{equation}

There exist $B\in M_{k , n-k}(F), D\in M_{n-k}(F)$ such that the block matrix
$$
\begin{bmatrix}
A & B \\
C & D
\end{bmatrix}
$$
has characteristic polynomial $f(x)$ if and only if $f_1(x)\cdots f_k(x) \mid f(x)$.
\end{theorem}
By applying Wimmer's theorem to the case where $f(x)$ is irreducible, we obtain the following result.
\begin{corollary}
\label{wimcor}
Let $A\in M_k(\Fq)$ and $C\in M_{n-k , k}(\Fq)$. The block matrix ${A \brack C} \in M_{n,k}(\Fq)$ can be completed to a matrix in $M_n(\Fq)$ with irreducible characteristic polynomial if and only if all the invariant factors of \eqref{xI-kcols} are equal to 1. 
\end{corollary}
 Corollary \ref{wimcor} establishes the equivalence between Problems \ref{howmanycompletions} and \ref{howmanyunimodular}.

\begin{definition}
We say that the ordered pair $(C,A)\in M_{n-k,k}(\Fq) \times M_k(\Fq)$ is a \emph{zero kernel pair} \cite[sec. X.1]{Gohbergetal1995} if 
$$
\bigcap_{i=0}^{k-1}\ker(CA^i)=\{\bf{0}\}.
$$
\end{definition}

\begin{remark}
The pair $(C,A)$ is called an \emph{observable pair} in the terminology of linear control theory.  
\end{remark}

%The condition that all the invariant factors of a polynomial matrix are equal to 1 is also expressed by saying that it is \emph{unimodular}.
The following proposition (\cite[Thm. IX.3.3]{Gohbergetal1995}, \cite[Thm. 23]{Sontag1998}) gives alternate characterizations of zero kernel pairs.
\begin{proposition}
\label{zerokernelpair}
Let $A\in M_k(\Fq)$, $C \in M_{n-k,k}(\Fq)$. The following are equivalent:
\begin{enumerate}
\item The product of the invariant factors of \eqref{xI-kcols} is 1.
% \item The largest $A$-invariant subspace contained in $\ker(C)$ is the zero subspace, i.e.,
% $$
% \bigcap_{i=0}^{k-1}\ker(CA^i)=\{\bf{0}\}.
% $$
\item The block matrix
  \begin{equation*}
    \begin{bmatrix}
      C \\
      CA \\
      \vdots \\
      CA^{k-1}
    \end{bmatrix}
  \end{equation*}
has rank $k$.
\end{enumerate}
\end{proposition}
It follows from Proposition \ref{zerokernelpair} that $(C,A)$ is a zero kernel pair if and only if $(A^T,C^T)$ is reachable, i.e., it satisfies the hypothesis of Problem \ref{numberofcpairs}. This is the well-known duality between reachable and observable pairs of matrices and establishes the equivalence of Problem \ref{numberofcpairs} with Problems \ref{howmanyunimodular} and \ref{howmanycompletions}. We explain the connection with Problem \ref{howmanyT} in the next section.

\section{Enumeration of simple linear transformations}
\label{enumeration}
Throughout this section, we denote by $V$ an $n$-dimensional vector space over $\Fq$.
\begin{definition}
\label{simpletransformation}
   Let $V$ be a vector space over a field $F$ and $W$ be a subspace of $V$. An $F$-linear transformation $T:W\to V$ is defined to be \emph{simple} if the only $T$-invariant subspace properly contained in $V$ is the zero subspace.
\end{definition}
This definition coincides with the usual definition of simple linear transformation (i.e., one with no nontrivial invariant subspaces) in the case $W=V$. Note that the definition does not allow the domain itself to be an invariant subspace for simple $T$ (unless the domain is $V$ or $\{\bf{0}\}$). %This definition is analogous to that of a simple linear operator on a vector space $V$ in the sense that simple operators are characterized by the fact that the only proper invariant subspace is $\{0\}$.

 To answer Problem \ref{howmanyT}, we need to count the number of simple $T:W\to V$ for a $k$-dimensional subspace $W$. In fact, the answer depends only on the dimension of $W$. To see this, let $W_1,W_2$ be distinct $k$-dimensional subspaces of $V$ and let $S:V\to V$ be a linear isomorphism such that $S(W_1)=W_2$. It is easily seen that $T:W_1\to V$ is simple if and only if $STS^{-1}:W_2\to V$ is simple. We thus have a bijection between simple maps with domain $W_1$ and those with domain $W_2$.

The following proposition relates simple maps to zero kernel pairs.
\begin{proposition}
  % Let $V,W$ be as above and let $\mathcal{B}_k=\{u_1,\ldots,u_k\}$ be an ordered basis of $W$. Extend $\mathcal{B}_k$ to an ordered basis $\mathcal{B}_n=\{u_1,\ldots,u_n\}$ of $V$. 
Let $V$ be an $n$-dimensional $\Fq$-vector space with ordered basis $\mathcal{B}_n=\{u_1,\ldots,u_n\}$. Let $\mathcal{B}_k=\{u_1,\ldots,u_k\}$ denote the ordered basis for the subspace $W$ spanned by $u_1,\ldots,u_k$. Then, a linear transformation $T:W\to V$ is simple if and only if the matrix of $T$ w.r.t. $\mathcal{B}_k$ and $\mathcal{B}_n$ is of the form ${A \brack C}$ for some zero kernel pair $(C,A)$.
\end{proposition}
\begin{proof}
 First suppose $T$ is not simple. Let $v$ be a nonzero vector lying in some $T$-invariant subspace and let the column vector $X_v\in \Fq^k$ denote the coordinates of $v$ w.r.t. $\mathcal{B}_k$. Let the matrix of $T$ w.r.t. the ordered bases $\mathcal{B}_k$ and $\mathcal{B}_n$ be 
\begin{equation}
\label{Tmatrix}
{A \brack C}  
\end{equation}
where $A\in M_k(\Fq)$ and $C\in M_{n-k,k}(\Fq)$. By the hypothesis, $T^iv$ lies in $W$ for all nonnegative integers $i$. The coordinate vector of $Tv$ w.r.t. $\Bn$ is
$$
{A \brack C}X_v={AX_v \brack CX_v}.
$$
Since $Tv\in W$, it follows that $CX_v=0$ and the coordinate vector of $Tv$ w.r.t. $\Bk$ is simply $AX_v$. By considering $T^2v$ it follows similarly that $CAX_v=0$ and the coordinate vector of $T^2v$ w.r.t. $\Bk$ is $A^2X_v$. Continuing this line of reasoning, we find that $CA^iX_v=0$ for $0\leq i \leq k-1$. Thus
$$
X_v\in \bigcap_{i=0}^{k-1}\ker(CA^i).
$$
Since $X_v\neq 0$, it follows that $(C,A)$ is not a zero kernel pair.

For the converse, let the matrix of $T$ w.r.t. $\Bk$ and $\Bn$ be \eqref{Tmatrix} as above and suppose $(C,A)$ is not a zero kernel pair. Then there exists a nonzero $X \in \bigcap_{i=0}^{k-1}\ker(CA^i)$. Then the vector $v\in W$ whose coordinate vector w.r.t. $\Bk$ is $X$ generates a nonzero $T$-invariant subspace. This completes the proof.
\end{proof}

The preceding proposition shows that counting simple maps is equivalent to counting zero kernel pairs. We have thus established the equivalence of all problems in the introduction. 

We now consider the solution of Problem \ref{howmanyT}. The number of $k$-dimensional subspaces of an $n$-dimensional vector space over $\Fq$ is given by the Gaussian binomial coefficient corresponding to $n$ and $k$:
$$
{n \brack k}_q=\frac{(q^n-1)(q^{n-1}-1)\cdots (q^{n-k+1}-1)}{(q^k-1)(q^{k-1}-1)\cdots (q-1)}.
$$

%Let $W_1,W_2$ be fixed subspaces of $V$ with $\dim W_1=\dim W_2=k$ and $\dim(W_1\cap W_2)=l$.  
For a fixed $k$-dimensional subspace $W$ of an $n$-dimensional vector space $V$ over $\Fq$, we define
\begin{equation*}
\psi_q(n,k):=\#\{T:W\to V : T \mbox{ is simple}\}.
\end{equation*}
We are thus interested in a formula for $\psi_q(n,k)$. Note that a simple map is necessarily of full rank. For subspaces $W,W'$ of $V$ with $\dim W=\dim W'$, we define
$$
\tau_q(W,W'):=\#\{T:W\to V: T\mbox{ is simple}, TW=W'\}.
$$
Counting simple maps by their image, it is easily seen that
\begin{equation*}
  \psi_q(n,k)=\sum_{U:\dim U=k}\tau_q(W,U),
\end{equation*}
where the sum is taken over all $k$-dimensional subspaces of $U$ of $V$. In fact, $\tau_q(W_1,W_2)$ depends only on $\dim W_1$ and $\dim(W_1\cap W_2)$. Hence, for $l\leq k$, we define 
$$
\tau_q(k,l):=\tau_q(W_1,W_2),
$$
where $W_1,W_2$ are any two $k$-dimensional subspaces with $\dim(W_1\cap W_2)=l$. %Let us denote by $\sigma_q(n,k,l)$ the number of $k$-dimensional subspaces of $V$ that intersect a fixed $k$-dimensional subspace of $V$ in a subspace of dimension $l$. %We have the following recurrence for $\tau_q(k,l)$.
% \begin{equation*}
% \sigma_q(n,k,l)&:=\#\{W \leq V: \dim W=k, \dim(W\cap W_1)=l\}.
% \end{equation*}
\begin{definition}
Let $V$ be an $n$-dimensional vector space over $\Fq$ and let $W$ be a fixed $k$-dimensional subspace of $V$. For $l\leq k$, define
$$
\sigma_q(n,k,l):=\#\{U: \dim U=k, \dim(U\cap W)=l\}.
$$
\end{definition}
\begin{proposition}
$\tau_q(k,l)$ satisfies the recurrence given by
  \begin{equation}
    \label{eq:taurec}
    \tau_q(k,l)=\sum_{m=0}^{l}\sigma_q(k,l,m)\tau_q(l,m)\prod_{i=l}^{k-1}(q^k-q^i).
  \end{equation}
\end{proposition}

\begin{proof}
  Any invariant subspace of a map $T:W_1\to W_2$ is necessarily contained in $W_1\cap W_2$. Therefore, to construct a simple $T:W_1\to W_2$, we first construct a simple map $T':W_1\cap W_2\to W'$ where $W'$ is some $l$-dimensional subspace of $W_2$. The number of such $W'$ for which $\dim(W_1\cap W_2\cap W')=m$ is $\sigma_q(k,l,m)$. For each $W'$ of this form, $T'$ can be constructed in $\tau_q(l,m)$ ways. $T'$ can be extended to a simple map $T:W_1\to W_2$ in $\prod_{i=l}^{k-1}(q^k-q^i)$ ways.
\end{proof}

\begin{proposition}
  The number of $k$-dimensional subspaces of $V$ that have an $l$-dimensional intersection with a fixed $k$-dimensional subspace of $V$ is given by
\begin{equation}
\label{eq:sigma}
\sigma_q(n,k,l)={k \brack l}_q {n-k \brack k-l}_q q^{(k-l)^2}.
\end{equation}
\end{proposition}
\begin{proof}
 Let $W_1$ be $k$-dimensional. We wish to count the number of $k$-dimensional $U$ for which $\dim(U\cap W_1)=l$. There are clearly ${k \brack l}_q$ choices for $U\cap W_1$. An ordered basis for $U\setminus (U\cap W_1)$ can be chosen in $\prod_{i=k}^{2k-l-1}(q^n-q^i)$ ways. Counting this way, the same $U$ arises in $\prod_{i=l}^{k-1}(q^k-q^i)$ ways. Thus the total number of such $U$ is given by
$$
\sigma_q(n,k,l)={k \brack l}_q \frac{\prod_{i=k}^{2k-l-1}(q^n-q^i)}{\prod_{i=l}^{k-1}(q^k-q^i)},
$$
which simplifies to the RHS of \eqref{eq:sigma}.
\end{proof}

The next proposition allows us to compute $\psi_q(n,k)$ from $\tau_q(k,l)$ and $\sigma_q(n,k,l)$.
\begin{proposition}
The number of simple linear transformations whose domain is a fixed $k$-dimensional subspace of $V$ is given by
  \begin{equation}
    \label{eq:psisigmatau}
 \psi_q(n,k)=\sum_{l=0}^{k} \sigma_q(n,k,l)\tau_q(k,l).    
  \end{equation}
\end{proposition}
\begin{proof}
Counting simple maps with domain $W$ of dimension $k$, we obtain
  \begin{align*}
    \psi_q(n,k)&=\sum_{U:\dim U=k}\tau_q(W,U)\\
               &=\sum_{l=0}^{k}\sum_{U:\dim(U\cap W)=l}\tau_q(W,U)\\
               &=\sum_{l=0}^{k}\sigma_q(n,k,l)\tau_q(k,l). \qedhere
  \end{align*} 
\end{proof} 
Note that, for $k=0$, $\tau_q(k,k)=1$ and for $k>1$, $\tau_q(k,k)=0$. On the other hand, for $k>0$, $\tau_q(k,0)=|\GL_k(\Fq)|$ where $\GL_k(\Fq)$ denotes the general linear group of $k\times k$ nonsingular matrices over $\Fq$. We define
$$\mu_q(k,l):=\frac{\tau_q(k,l)}{\prod_{i=1}^{k-1}(q^k-q^i)}.$$ 
Then the recurrence \eqref{eq:taurec} becomes
\begin{align*}
\mu_q(k,l)&=\frac{\prod_{i=1}^{l-1}(q^l-q^i)}{\prod_{i=1}^{l-1}(q^k-q^i)}\sum_{m=0}^{l}\sigma_q(k,l,m)\mu_q(l,m)  \\
          &=\frac{1}{{k-1 \brack l-1}_q}\sum_{m=0}^{l}\sigma_q(k,l,m)\mu_q(l,m).
\end{align*}

\begin{proposition}
\label{lem:mu}
We have
  \begin{equation}
    \label{eq:mu}
  \mu_q(k,l)=
\begin{dcases}
1 & k=l=0,\\
q^k-q^l & \mbox{otherwise}.
\end{dcases}
  \end{equation}
\end{proposition}
\begin{proof}
The formula for $\mu_q(k,l)$ is easily verified for $k=0,1$ directly from the definition of $\tau_q(k,l)$. We use induction on $k$. Suppose $k>1$. Clearly $\mu_q(k,k)=0$ and $\mu_q(k,0)=q^k-1$, so suppose $0<l<k$. Then
\begin{align*}
{k-1 \brack l-1}_q  \mu_q(k,l)&=\sum_{m=0}^{l}\sigma_q(k,l,m)\mu_q(l,m)\\
            &=\sum_{m=0}^{l}{l \brack m}_q {k-l \brack l-m}_q q^{(l-m)^2}(q^l-q^m)\\
            &=\sum_{m=0}^{l}{k-l \brack l-m}_q q^{(l-m)^2}q^m(q^l-1){l-1 \brack m}_q\\
            &=q^l(q^l-1)\sum_{m=0}^{l}{k-l \brack l-m}_q {l-1 \brack m}_q q^{(l-m)(l-m-1)}\\
            &=q^l(q^l-1){k-1 \brack l}_q,
\end{align*}
where the last equality follows from the $q$-Vandermonde identity \cite[Thm. 3.4]{Andrews}. It follows that $\mu_q(k,l)=q^k-q^l$ as desired.
\end{proof}
We are now ready to prove the main theorem of this paper.
\begin{theorem}
\label{main}
We have
$$
  \psi_q(n,k)=\prod_{i=1}^{k}(q^n-q^i).
$$
\end{theorem}
\begin{proof}
From \eqref{eq:taurec} and \eqref{eq:psisigmatau}, it is easily seen that
\begin{equation*}
\tau_q(n,k)=\psi_q(n,k)\prod_{i=k}^{n-1}(q^n-q^i).
\end{equation*}
Substituting $\tau_q(n,k)=(q^n-q^k)\prod_{i=1}^{n-1}(q^n-q^i)$, the theorem follows.
\end{proof}
\begin{corollary}
  The number of zero kernel pairs $(C,A) \in M_{n-k,k}(\Fq) \times M_k(\Fq)$ is given by $\prod_{i=1}^{k}(q^n-q^i).$
\end{corollary}
\begin{corollary}
  The number of reachable pairs $(A,B) \in M_{k,k}(\Fq) \times M_{k,n-k}(\Fq)$ is given by $\prod_{i=1}^{k}(q^n-q^i).$
\end{corollary}
Theorem \ref{main} gives us the following answer to Problem \ref{howmanyunimodular}.
\begin{corollary}  
\label{probofunimod}
  The number of matrices $Y\in M_{n,k}(\Fq)$ for which 
  \begin{equation}
    \label{eq:unimod}
x{I_k \brack \bf{0}}-Y    
  \end{equation}
is unimodular is given by $\prod_{i=1}^{k}(q^n-q^i).$
\end{corollary}
\section{Conclusion}
 Corollary \ref{probofunimod} shows that for a uniformly random $Y\in M_{n,k}(\Fq)$, the probability that \eqref{eq:unimod} is unimodular is given by $\delta_q(n,k)=\prod_{i=1}^{k}(1-q^{i-n})$. Guo and Yang \cite[Thm. 1]{GuoYang2013} have recently proved that the probability that a uniformly random $n\times k$ matrix over $\Fq[x]$ is unimodular is also given by $\delta_q(n,k)$.  Corollary \ref{probofunimod} shows that the probability of being unimodular is unaltered even if we consider a very specific subset of $M_{n,k}(\Fq[x])$. It would thus be interesting to find and characterize other subsets of rectangular polynomial matrices for which this property remains true. For a positive integer $m$ and $k<n$, we define  
$$
M_{n,k}(\Fq[x];m):=\left\{x^m{I_k \brack \bf{0}}+\sum_{i=0}^{m-1}x^iA_i: A_i\in M_{n,k}(\Fq) \mbox{ for } 0\leq i \leq m-1\right\}.
$$
 We have the following conjecture on the density of unimodular matrices.
\begin{conjecture}
\label{genunimod}
  The probability that a uniformly random element of $M_{n,k}(\Fq[x];m)$ is unimodular is given by $\delta_q(n,k)$.
\end{conjecture}
We remark that the conjectured probability is independent of $m$ (the common degree of the matrix polynomials). Conjecture \ref{genunimod} is of great interest since %it may be viewed, in a sense, as a  generalization of the problems in the introduction. 
a solution would provide an alternate resolution of the problems stated in the introduction.
While the conjecture has been verified using computer programs for small values of $m,n,k,q$, the general case still appears to be open.

\section*{Acknowledgements}
The author would like to thank Prof. Arvind Ayyer for some discussions and his careful reading of a preliminary version of this paper. This work is supported partly by a UGC Center for Advanced Study grant and partly by a DST Centre for Mathematical Biology phase II grant.
\bibliographystyle{plain}   
\bibliography{biblio}

\end{document}